\documentclass[english,12pt]{article}

\usepackage{babel}
\usepackage[latin1]{inputenc}
\usepackage{latexsym}
\usepackage{amsmath}
\usepackage{amsfonts}
\usepackage{theorem}
\usepackage{float}
\usepackage{subcaption}
\usepackage{relsize}
\usepackage{xcolor}
\usepackage{url}
\usepackage{graphics}

\usepackage[noline, boxed]{algorithm2e}

\SetCommentSty{mycommfont}

\usepackage{multirow}

\usepackage{graphicx}                 
\usepackage{hyperref}

\usepackage[authoryear]{natbib}
 
   \newcounter{prop}
   
 \newcounter{theo}

\newtheorem{lemma}{Lemma}

\newtheorem{example}{Example}

\newcommand{\nc}{\newcommand}
\nc{\vekk}[1]{}

\nc{\lik}{ = }  \nc{\supp}{\mbox{supp}} \nc{\E}{\mbox{E}}

\nc{\ttau}{\tilde{\tau}} \nc{\ttheta}{\tilde{\theta}}

\nc{\fra}{From} \nc{\st}{|}

\newcommand{\bX}{{\mathbf{X}}}
\newcommand{\bx}{{\mathbf{x}}}

\nc{\gof}{goodness-of-fit }
\nc{\Gof}{Goodness-of-fit }

\nc{\beq}{\begin{equation}}

\nc{\eeq}{\end{equation}}

\nc{\K}{K}

\nc{\beqns}{\begin{eqnarray*}}
\nc{\eeqns}{\end{eqnarray*}}
\nc{\beqn}{\begin{eqnarray}}
\nc{\eeqn}{\end{eqnarray}}

\nc{\hatt}{\hat{\theta}}

\nc{\beit}{\begin{itemize}}
\nc{\eit}{\end{itemize}}

\newenvironment{proof}{\paragraph{Proof:}}{\hfill$\square$}

\begin{document}

\def\spacingset#1{\renewcommand{\baselinestretch}%
	{#1}\small\normalsize} \spacingset{1}


{
	\title{\bf Conditional Goodness-of-Fit Tests for Discrete Distributions}
	\author{Rasmus Erlemann\footnote{Email: rasmus.erlemann@ntnu.no\ \ \ \ \ Address: Kilu 17, 13516, Tallinn, Estonia},\ Bo Henry Lindqvist\\
		\small{Department of Mathematical Sciences, NTNU}
	}
	\maketitle
}
\thispagestyle{empty}

\bigskip
\begin{abstract}
	In this paper, we address the problem of testing goodness-of-fit for discrete distributions, where we focus on the geometric distribution. We define new likelihood-based goodness-of-fit tests using the beta-geometric distribution and the type I discrete Weibull distribution as alternative distributions. The tests are compared in a simulation study, where also the classical goodness-of-fit tests are considered for comparison. Throughout the paper we consider conditional testing given a minimal sufficient statistic under the null hypothesis, which enables the calculation of exact $p$-values. For this purpose, a new method is developed for drawing conditional samples from the geometric distribution and the negative binomial distribution. We also explain briefly how the conditional approach can be modified for the binomial, negative binomial and Poisson distributions. It is finally noted that the simulation method may be extended to other discrete distributions having the same sufficient statistic, by using the Metropolis-Hastings algorithm.  
\end{abstract}

\noindent%
\vfill

\newpage

\section{Introduction}
\spacingset{1.5}
\indent Let $X_1,X_2,\ldots ,X_n$ be a random sample from a distribution $F$. Goodness-of-fit testing is concerned with how well a family of distributions $\mathcal{F}$ fits the data as a probability model. The null hypothesis is $F\in \mathcal{F}$ and the alternative hypothesis is $F\notin \mathcal{F}$. In-depth literature on this topic includes \cite{agostino}. In the literature of goodness-of-fit testing, most of the work has been focused on continuous distributions, i.e. $\mathcal{F}$ a family of continuous distributions. For discrete distributions, the main interest has been in the Poisson distribution which plays a special role in probability theory. Goodness-of-fit tests for the Poisson distribution go at least back to \cite{fisher1950} and \cite{rao1956}. More recent studies of the Poisson distribution are \cite{LockhartPoisson} and \cite{RPoisson}. 

Common alternatives to the Poisson distribution are the negative binomial distribution and its special case, the geometric distribution. 
The latter distribution is of particular interest since it is the discrete counterpart of the exponential distribution, and is hence an important distribution with various applications, for example in survival analysis, reliability analysis and queuing theory. 
\cite{BCG} presented a comprehensive study of different goodness-of-fit test statistics for the geometric distribution and a comparison between them in a simulation study. Another paper considering tests for the geometric distribution is  \cite{ODG}. The present paper will mainly be concerned with \gof testing for the geometric distribution, although several ideas considered can easily be modified to cover other discrete distributions.
 
Traditional goodness-of-fit tests, both in the continuous and discrete distribution cases, are the Kolmogorov-Smirnov, Cram\'er-von Mises and Anderson-Darling tests, see for example \cite{agostino}. Various methods are used for finding critical values, typically based on standard asymptotic techniques or parametric bootstrapping. There are, however, also other methods or tricks available, often used to tailor goodness-of-fit testing for specific models.

One such ``trick'', which will be the main tool in the present paper, is to condition on sufficient statistics under the null hypothesis model to be tested. Such approaches go back to the 1950s. More specifically, \cite{fisher1950} obtained in this way exact versions of the chi-squared test and an alternative test based on the dispersion for the Poisson distribution, using the fact that the sum of the observations is a sufficient statistic in this case. As a follow-up, \cite{rao1956} used the same idea to derive an exact test for the Poisson case based on a likelihood ratio statistic (see Section~\ref{full}). Conditioning on sufficient statistics has also been used recently in \cite{reillypoisson} and  \cite{puig}. While the just cited papers have considered models with one unknown parameter under the null hypothesis, \cite{heller1986} did goodness-of-fit testing for the two-parameter negative binomial distribution, assuming both parameters are unknown. Then she conditioned on the sum of the observations in order to eliminate the probability parameter and then using an asymptotic approach having only one unknown parameter. 

Often, the sufficient statistic under the null model is easy to find, but still the calculation of critical values or $p$-values for the conditional tests can be problematic. Usually it will be necessary to sample from the conditional distributions given the sufficient statistic. For goodness-of-fit testing in continuous distributions, and in particular in models where there are more than one parameter, this may however not be straightforward. For possible approaches, see \cite{LT05}, \cite{CMCr}, \cite{lockhartgibbs}.

For the most common discrete distributions, like the binomial and the Poisson distribution, it is straightforward and well known how to do conditional sampling \citep{gonzalez}. How to perform conditional sampling for the geometric distribution and the negative binomial distribution is, apparently, less studied. \cite{gonzalez} derive the conditional distribution for this case, but does not advice a way of simulating from it. In Section~\ref{sec:condgeo} we show how this can be done by using the so called ``bars and stars'' framework of \cite{feller1}. It is believed that the associated algorithm is new in goodness-of-fit studies of the geometric distribution. An extension to the negative binomial distribution is given in the Appendix. Another way of obtaining conditional samples in discrete distributions is suggested by \cite{reillypoisson}, based on the so called Rao-Blackwell distribution. 

The most important ingredient of a goodness-of-fit test is of course the test statistic. The three standard tests, the Kolmogorov-Smirnov, the Cram\'er-von Mises and the Anderson-Darling test, are already mentioned. These are examples of tests based on the empirical distribution function of the data. While the Kolmogorov-Smirnov statistic considers the maximal difference between the null model and the empirical distribution of the data, the two other tests are based on the corresponding integrated squared difference.  A well known fact is that the Anderson-Darling statistic differs from the  Cram\'er-von Mises statistic in that it gives more weight to extreme values of the observations. There are in the literature also considered other test statistics that are known to be large  when the null hypothesis  model does not hold, but without connection to particular alternative models. Examples are chi-squared tests and tests based on Fisher's index of dispersion, which is the ratio of the variance to the mean,  and is well known to be 1 for the Poisson distribution. Closely related to these tests are the tests derived by \cite{kyria} based on characterizations of the Poisson, binomial and negative binomial distributions by their power-series representations (see Section~\ref{betageo}). 

The above tests are essentially not tailored for specific alternative distributions. There might in applications be of importance, however, to have tests that are particularly powerful for given alternative distributions. One purpose of the present paper is to investigate how well the standard \gof tests for the geometric distribution will do compared to tests tailored for specific alternatives. 
	
A classical problem is to test the Poisson distribution versus models for over-dispersion, such as the negative binomial distribution. The above cited paper by \cite{puig} gives another example. These authors considered testing of the Poisson distribution versus alternatives with log-convex probability generating functions, shown to have important applications in biodosimetry. 

 \cite{weinberg}) considered human fecundability data, using the geometric distribution to model the number of menstrual cycles required to achieve pregnancy. It is then reasonable to believe that the parameter $p$ of the geometric distribution varies between couples. The cited authors showed how to model this variation by means of beta distributions, which leads to counts following a so called beta-geometric distribution. Subsequently, \cite{paul} studied goodness-of-fit testing for the geometric distribution by testing versus the alternative being the beta-geometric distribution, using a score test and a likelihood ratio test. We return to this in Section \ref{betageo}. 
 
In reliability, a classical problem is to test the null hypothesis of an exponential distribution versus the alternative of a Weibull distribution. This may be done in a straightforward manner using a likelihood ratio test. In the discrete case, this would mean to test the geometric distribution versus some kind of discrete Weibull distribution. We will consider this problem in Section \ref{Weibulldisc}, using the so called type I discrete Weibull distribution, see for example \cite{BG}. 

As a final comment on the use of conditional testing, one might ask what is possibly lost in power by such an approach when compared to unconditional ones. We have not pursued this problem, but refer to  \cite{lockhartreilly}  who concluded from a particular study that calculated $p$-values from conditional tests are highly correlated with $p$-values found by parametric bootstrapping. An apparent advantage with the conditional tests is, moreover, that these tests are exact, while the bootstrap based tests are not exact (albeit almost so).

In the second section, we introduce how conditional tests are used in goodness-of-fit testing with discrete null hypothesis. We also cover how the $p$-values and powers are calculated with Monte Carlo methods in that setting. In the third section we present our method for drawing conditional simulations from the geometric distribution. It is based on the so-called stars and bars representation, introduced by \cite{feller1}. In the same section, we also introduce some classical test statistics and define new likelihood based tests. In the end of the section, there are two examples where the data is simulated from the beta-geometric and the discrete Weibull distribution of type I and we calculate the conditional $p$-values for both cases. The fourth section consists of the power study and simulated type I errors. In the fifth section we consider a real life data set and use previously mentioned methods to test if the geometric distribution fits the data. The sixth section consists of conclusions and a brief outline for possible future work in this subject. The last section is the appendix. There are proofs, algorithm descriptions, method for the negative binomial and parameterizations of the distributions we used. References are at the very end of this paper.

In the following we shall let $\mathbb{N}=\{1,2,\ldots\}$ and $\mathbb{N}_0=\{0,1,2,\ldots \}$. A sample of random variables $Y_1,Y_2,\ldots ,Y_n$ is denoted shortly in its vector form by a bold letter, $\mathbf{Y}=(Y_1,Y_2,\ldots ,Y_n)$. Bold $\mathbf{P}$ is reserved for the probability function to differentiate it from other functions. Notation for different distributions and the parametrizations are specified in the appendix. 

\section{Conditional Tests in Goodness-of-Fit Testing for Discrete Distributions}
\label{MC}

\subsection{Calculation of Conditional $p$-values by Monte Carlo Simulation}
\label{MCsub}

For illustration  we focus on \gof testing for the geometric distribution. Suppose $\bX=(X_1,\ldots,X_n)$ is a random sample from a population with values in $\mathbb{N}_0$.  The null and alternative hypotheses are as follows.
\begin{align*}H_0 &\ \colon \ \text{The random sample is from a population which has the geometric distribution,}\\
H_1&\ \colon \ \text{The random sample is not from a population which has the geometric distribution. }\end{align*}

Let $D=D(\bX)$ be a test statistic such that large values of $D$ are supposed to indicate deviations from $H_0$. Suppose that, under the null hypothesis, $T=T(\bX)$ is a sufficient statistic. Algorithm~\ref{condp} in the Appendix calculates, by Monte Carlo simulation,  the conditional $p$-value of the test from the formula
\[
   p^{\text{cond}} = \mathbf{P}(D(\bX) \ge D(\bx_{obs}) \;| \;T(\bX)=t)
\]
where $\bx_{obs}$ is the observed value of $\bX$ and $t$ is the observed value of $T(\mathbf{X})$. For the Monte Carlo simulation one therefore needs a way of simulating from the conditional distribution of $\bX$ given $T(\bX)=t$. In the next Section we show how this can be done in the case of the geometric distribution. 

\subsection{Calculation of Test Power of Conditional Tests}
To calculate the power of a \gof test for a given alternative distribution and for a given significance level $\alpha$, we proceed as follows. Draw a large number $M$ data sets from the alternative distribution. For the $i$-th such set, calculate the conditional $p$-value, $p_i^{\text{cond}}$ by Algorithm \ref{condp}, $i=1,2,\ldots ,M$. The Monte Carlo power of the test can then be calculated as
\[
\beta(\alpha) \approx 
\frac{\sum_{i=1}^M I(p_i^{\text{cond}} \le \alpha)}{M}.
\]

Power calculations require a large number of iterations. Let $K$ be the number of iterations used to calculate each conditional $p$-value. If $M$ is the number of data sets drawn to calculate the power, then in total we are doing $M$ times $K$ iterations. The number of data sets $M$ is chosen to be as large as possible depending on computational capabilities. We chose $M=1000$.

It should be noted that since we are dealing with discrete distributions, for a given data set $\bx$ with $T(\bx)=t$, there are only finitely many possible data sets in the conditional distribution of $\bX$ given $T(\bX)=t$. This means that, although we fix a significance level $\alpha$ and are guaranteed a size of the test that is $\le \alpha$, the size may be strictly less than $\alpha$. This problem is of course of less concern if $n$ is large. 
 
 \section{\Gof Testing in the Geometric Distribution}
 
\subsection{Conditional Sampling from the Geometric Distribution}
\label{sec:condgeo}
Let $X_1,X_2, \ldots ,X_n$ be iid random variables, such that  $X_i \sim \text{Geom}(p)$ for all $i=1,2,\ldots ,n$, i.e.,
\[
\mathbf{P}(X_i=x) = p(1-p)^{x} \; \mbox{ for } x=0,1,2,\ldots
\]
 Then $T(\mathbf{X})=\sum_{i=1}^nX_i=t$ is a sufficient statistic. The conditional distribution of $\bX$ given $T(\mathbf{X})=t$ is calculated as
\begin{align*}
\mathbf{P}(\mathbf{X}=\mathbf{x}\; |\; T(\mathbf{X})=t)&= \mathbf{P}(X_1=x_1, X_2=x_2, \ldots ,X_n=x_n\; | \; \sum_{i=1}^nX_i=t)\\
&=\frac{\mathbf{P}(X_1=x_1,X_2=x_2,\ldots ,X_n=x_n, \sum_{i=1}^nX_i=t)}{\mathbf{P}(\sum_{i=1}^nX_i=t)}\\
&=\left\{\begin{array}{lr}\displaystyle{\frac{\mathbf{P}(X_1=x_1,X_2=x_2,\ldots ,X_n=x_n)}{\mathbf{P}(\sum_{i=1}^nX_i=t)}}, & \text{ if } \sum_{i=1}^nX_i=t\\ 0, & \text{ if } \sum_{i=1}^nX_i\neq t\end{array}\right. .
\end{align*}
If we restrict the support to be  $S=\{ (x_1,x_2,\ldots ,x_n)\; \colon \; \sum_{i=1}^nx_i=t, \ x_1,x_2,\ldots ,x_n\in \mathbb{N}_0 \}$, we get
\begin{align}
\mathbf{P}(\mathbf{X}=\mathbf{x}\; |\; T(\mathbf{X})=t)&=\frac{p^n(1-p)^{\sum_{i=1}^nx_i}}{p^n(1-p)^t\binom{t+n-1}{n-1}}\nonumber\\
&=\frac{1}{\binom{t+n-1}{n-1}}.\label{condprob}
\end{align}
As the conditional probability \eqref{condprob} does not depend on $\mathbf{x} \in S$, the distribution of $\mathbf{X}\; | \; T(\mathbf{X})=t)$ is uniform on $S$, i.e., on the set of all possible ways that $n$ non-negative integers sum to $t$. \cite{feller1} introduced a representation of such sums through the so called ``stars and bars'' framework. To construct one such sum, we lay down $t$ stars and put $n-1$ bars between them.  The sum is constructed by counting the number of stars between the bars, letting the first element be the number of stars in front of the first bar, and letting the last one be the number to the right of the last bar. For example, for $t=8$ and $n=4$, Figure \ref{starsandbars} represents the sum $1+0+3+4=8$. 
\begin{figure}[H]
	\begin{center}
		\includegraphics[width=150mm]{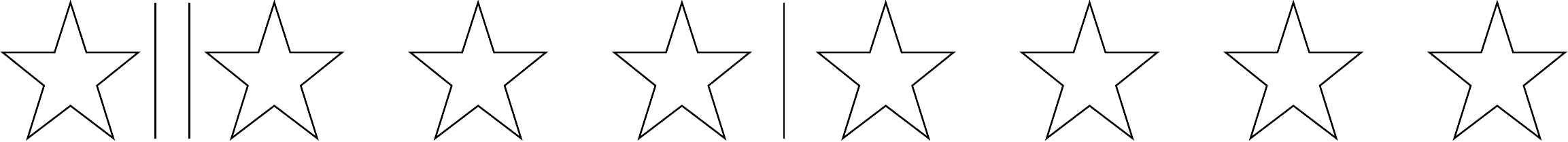}
	\end{center}
	\caption{Stars and Bars Representation}
	\label{starsandbars}
\end{figure}
  The following lemma states that the representation gives rise to any possible sum and conversely. The proof is given in Appendix. 
\begin{lemma}
\label{biject}
For $t,n\in \mathbb{N}$, let
	 	$$L_1=\left\{ (x_1,x_2,\ldots ,x_n)\; \colon \; \sum_{i=1}^nx_i=t, \ x_1,x_2,\ldots ,x_n\in \mathbb{N}_0 \right\}$$
	and $$L_2=\left\{ (k_1,k_2,\ldots ,k_{n-1}) \; \colon \; k_1 < k_2 < \ldots < k_{n-1} < t+n,\ k_1,\ldots ,k_{n-1}\in \mathbb{N}\right\}.$$	
		Define a transformation $\phi\colon L_2\to L_1$, such that
	$$\phi(k_1,k_2,\ldots ,k_{n-1})=$$
	$$=(k_1-1, (k_2-2)-(k_1-1),\ldots ,(k_{n-1}-(n-1))-(k_{n-2}-(n-2)), t-(k_{n-1}-(n-1))).$$
	Then $\phi$ is a bijection between the sets $L_2$ and $L_1$. 
	\end{lemma}
	
\begin{example}
The $k_j$ in $L_2$ are the positions of the bars in the stars and bars representation. For example, the stars and bars representation of $n=4$ and $t=1+0+3+4=8$ in Figure \ref{starsandbars},   correspond to	$$k_1=2, \ k_2=3, \ k_3=7.$$
	\end{example}

\begin{lemma}\label{many}
	Let $t,n\in \mathbb{N}$. The set
	$$\left\{ (x_1,x_2,\ldots ,x_n)\; \colon \; \sum_{i=1}^nx_i=t, \ x_1,x_2,\ldots ,x_n\in \mathbb{N}_0 \right\}$$
	has $\binom{t+n-1}{n-1}$ elements.
	\end{lemma}
The proof of Lemma \ref{many} can be found in \cite{feller1}. In fact, it also follows from \eqref{condprob}.

Algorithm RandomKSubsets from \cite{Herb} is a method for drawing the so-called bars $k_1,\ldots ,k_{n-1}$ uniformly from $L_2$. We have modified it by making recursive calls into iterative ones. This allows the algorithm to be used with large values of $n$ and $t$ more efficiently and there are no issues with recursion depth limitations. It is described by Algorithm \ref{norec} in Appendix, where a proof of its correctness is also given (Lemma~\ref{correct}).

Algorithm \ref{norec} gives us a sample $k_1,\ldots , k_{n-1}$ and the last step is to transform it back into an element from $L_1$. We use the previously defined function $\phi$ for it and
$$\phi(k_1,k_2,\ldots ,k_{n-1})\sim \mathbf{X}\; | \; T(\mathbf{X})=t.$$

\subsection{Standard \Gof Test Statistics for Discrete Distributions}
The following is a general setup for calculation of test statistics for the most common goodness-of-fit tests, with focus on the geometric distribution. The setup essentially follows the one of \cite{LockhartPoisson}
who in particular studied the performance for the Poisson distribution.

Let $x_1,x_2\ldots ,x_n$ be the observed sample and $t=\sum_{i=1}^nx_i$ the sufficient statistic. Maximum likelihood estimator for the geometric distribution is given by 
$$\hat{p}=\frac{n}{t+n}.$$

In order to avoid trivial cases, we will assume $t>0$ and hence $0<\hat{p}<1$. Now, define for $j=0,1,2,\ldots$,
\begin{eqnarray*}
	o_j &=& \# \{ i: x_i=j\} = \mbox{ observed number of values $j$ for the sample}\\
	\hat p_j &=& \hat p(1-\hat p)^j = \mbox{ probability of value $j$ in geometric distribution}\\
	\hat e_j &=& n \hat p_j= \mbox{ estimated expected number of values $j$ for the sample}
\end{eqnarray*}
From this define, for $k=0,1,2,\ldots$,
\begin{eqnarray*}
	\hat Z_k &=& \sum_{j=0}^k (o_j - \hat e_j) \equiv O_k - \hat E_k \\
	\hat H_k &=& \sum_{j=0}^k \hat p_j
\end{eqnarray*}
where $O_k = \sum_{j=0}^k o_j$ is the observed number of values $\le k$ in the sample and $\hat E_k = n \hat H_k$ is its estimated expected value. 

Further, define
\begin{eqnarray*}
	M_0^u &=& \min\{ j: o_{j'}=0 \mbox{ for all } j' > j\} \\
	M_1^u &=& \min\{ j: p_{j'} < 10^{-3}/n \mbox{ for all } j' > j\} \\
	M^u &=& \max\{ M_0,M_1\} \\
	M_0^l &=& \max\{ j: o_{j'}=0 \mbox{ for all } j' < j\} \\
	M_1^l &=& \max\{ j: p_{j'} < 10^{-3}/n \mbox{ for all } j' < j\} \\
	M^l &=& \min\{ M_0,M_1\}
\end{eqnarray*} 

\subsubsection{The Cram\'{e}r-von Mises Test}
The Cram\'{e}r-von Mises test statistic is defined by
$$W^2=\frac1n \sum_{M^{l}}^{M^{u}}\hat{Z}_i^2\hat{p}_i.$$

\subsubsection{The Anderson-Darling Test}
The Anderson-Darling test statistic is defined by
$$A^2 = \frac1n\sum_{M_l}^{M^u}\frac{\hat{Z}_i^2\hat{p}_i}{\hat{H}_i(1-\hat{H}_i)}.$$

\subsubsection{The Kolmogorov-Smirnov Test}
The Kolmogorov-Smirnov test statistic should ideally be defined as 
$\max_{k=1,2,\ldots}|Z_k|.$ 
As shown by \cite{BCG}, the maximum will always occur for a $k \le M_0^u$, so that we define
$$KS=\max_{k=0,1,2,\ldots,M_0^u}|Z_k|.$$
To see this, recall that $\hat Z_k = O_k - \hat E_k$. Now $O_k = n$ for $k \ge M_0^u$, while $\hat E_k < n$ and $\hat E_k$ is increasing in $k$. Hence, for $k \ge M_0$, $|Z_k|=Z_k$ and is decreasing.

\subsection{Likelihood Based Tests}
In the present subsection we study tests that are derived with the aim of having high power against given alternative distributions. The main tool is here to consider likelihood functions.

\subsubsection{Test Versus Heterogeneous Geometric Observations}
\label{full}
In this subsection we follow the idea of \cite{rao1956}, who considered the Poisson distribution where we consider the geometric distribution.

Suppose $X_1,X_2,\ldots,X_n$ are independent and geometrically distributed, but with different parameters $p_i$. The log likelihood for  data $x_1,x_2,\ldots,x_n$ would then be
\[
\ell(p_1,p_2,\ldots,p_n) = \sum_{i=1}^n (\log(p_i) + x_i \log (1-p_i)),
\]
which is maximized by $\hat p_i = 1/(1+x_i)$ for $i=1,2,\ldots,n$. 
The relevant null hypothesis is now
\[
H_0: p_1=p_2= \ldots = p_n = p.
\] 
The log likelihood under the null hypothesis is then $\ell(p,\ldots,p) = n \log(p) + t \log(1-p)$ where $t=\sum_{i=1}^n x_i$, which is maximized by $\hat p=n/(n+t)$.

The likelihood ratio statistic can therefore be written
\begin{eqnarray*}
	&&\ell(\hat p_1, \hat p_2,\ldots,\hat p_n)  - \ell(\hat p,\ldots,\hat p)\\
	&=& \sum_{i=1}^n \left[ \log\left(\frac{1}{1+x_i}\right) + x_i  \log\left(\frac{x_i}{1+x_i}\right) \right] -n \log\left(\frac{n}{n+t}\right) -t  \log\left(\frac{t}{n+t}\right) \\
	&=&  \sum_{i=1}^n \left[ x_i \log(x_i)-(x_i+1)\log(x_i+1) \right] -n \log\left(\frac{n}{n+t}\right) -t  \log\left(\frac{t}{n+t}\right)
\end{eqnarray*}
Since we consider conditional tests given $\sum_{i=1}^n X_i =t$, we may exclude the last terms above, which after rewriting the first sum gives the test statistic
\[
CR = \sum_{i=1}^n \left[ x_i \log(x_i)-(x_i+1)\log(x_i+1) \right] = \sum_{j=M_0^l}^{M_0^u} o_j\big(j \log j - (j+1)\log(j+1)\big)
\] 
where we use $0\log 0=0$.

\subsubsection{The Beta-Geometric Distribution} \label{betageo}
In the previous subsection we considered the alternative hypothesis that the observations were geometrically distributed, but with possibly different parameters $p_i$. Suppose now that these $p_i$ are drawn independently from the beta distribution. 

Thus, for a single observation $X$ we assume that it is geometrically distributed with parameter $p$, where $p$ is generated from the beta-distribution with parameters $\alpha>0$ and $\beta>0$. Let $B(\alpha,\beta)$ be the beta function, defined by
\[
B(\alpha,\beta) = \int_0^1  p^{\alpha-1}(1-p)^{\beta-1} dp = 
\frac{\Gamma(\alpha)\Gamma(\beta)}{\Gamma(\alpha+\beta)}.
\]
Then it is seen that the unconditional distribution of $X$ is what has been named the beta-geometric distribution,
\begin{equation}
\label{fab}
\mathbf{P}(X=x) = \int_0^1 p(1-p)^x  \
\frac{p^{\alpha-1}(1-p)^{\beta-1}}{B(\alpha,\beta)}dp
=\frac{B(\alpha+1,\beta+x)}{B(\alpha,\beta)}
\end{equation}
for $x=0,1,2,\ldots$. As suggested by \cite{paul}, a useful reparametrization is 
\beq
\label{repar}
\pi = \frac{\alpha}{\alpha+\beta}, \; \; \theta = \frac{1}{\alpha+\beta}.
\eeq
With this parametrization it is seen that $\theta=0$ corresponds to the geometric distribution with $p=\pi$. Tests for the null hypothesis of geometric distribution can hence be derived by testing
\[
H_0: \theta = 0 \mbox{ vs. } \theta > 0.
\]
Using the reparametrization (\ref{repar}), we find from (\ref{fab}), noting that $\alpha=\pi/\theta$, $\beta=(1-\pi)/\theta$ and using properties of the gamma function,
\begin{equation}
\mathbf{P}(X=x) =  \frac{\alpha \prod_{j=1}^x (\beta +x-j)}{\prod_{j=0}^x (\alpha + \beta +x-j)} =  \frac{\pi \prod_{j=0}^{x-1} (1-\pi +j \theta)}{\prod_{j=0}^x (1+j \theta)}
\end{equation}
for $x=0,1,\ldots$. (Note that the formula also holds for $x=0$, giving $\mathbf{P}(X=0)=\pi$,  since an empty product by convention equals 1.)

The log-likelihood for data $x_1,x_2,\ldots,x_n$ with values in $\mathbb{N}_0$ is hence, see also \cite{paul},
\begin{equation}
\label{loggl}
\ell(\bx;\pi,\theta) = n \log\pi + \sum_{i=1}^n \sum_{j=0}^{x_i-1} \log(1-\pi+j\theta) - \sum_{i=1}^n \sum_{j=0}^{x_i} \log(1+j\theta)
\end{equation}
Differentiating with respect to $\theta$ and $\pi$ give, respectively,
\begin{eqnarray*}
	\frac{\partial \ell}{\partial \theta}  &=&  \sum_{i=1}^n \sum_{j=0}^{x_i-1} \frac{j}{1-\pi+j\theta} - \sum_{i=1}^n \sum_{j=0}^{x_i} \frac{j}{1+j\theta}, \\
	\frac{\partial \ell}{\partial \pi} & =& \frac{n}{\pi} -  \sum_{i=1}^n \sum_{j=0}^{x_i-1} \frac{1}{1-\pi+j\theta}.
\end{eqnarray*}

Letting $\theta=0$ in the expression for $\frac{\partial \ell}{\partial \theta}$ gives the score statistic for testing $H_0: \theta=0$, namely
\beq
\label{SS}
S  =  \sum_{i=1}^n \sum_{j=0}^{x_i-1} \frac{j}{1-\pi} - \sum_{i=1}^n \sum_{j=0}^{x_i} j \; =\;  \frac{\pi \sum_{i=1}^n x_i^2 - (2-\pi)\sum_{i=1}^n x_i}{2(1-\pi)}.
\eeq
The score test in general rejects $H_0: \theta=0$ for large values of $|S|$. Indeed, it can be shown from the rightmost expression in (\ref{SS}) that, under $H_0$ where the $X_i$ are
geometrically  distributed with probability $p=\pi$,  we have $E(S)=0$.  
\cite{paul} replaced $\pi$ by the maximum likelihood estimate $\hat p$ under $H_0$, and divided the expression in (\ref{SS})  by an estimate of the standard deviation of $S$ under $H_0$, which is $\sqrt{n}/\hat p$. The resulting statistic then has an asymptotically  standard normal distribution under $H_0$. 

Let now $m_1=(1/n)\sum_{i=1}^n x_i$ and $m_2=(1/n)\sum_{i=1}^n x_i^2$ be the first and second empirical moments, respectively, from the data $\bx$.  
Replacing $\pi$ by $\hat p = n/(t+n)=1/(1+m_1)$ we can write the right hand side of (\ref{SS}) as 
\[
n \frac{ \frac{m_2}{1+m_1} - \left(2-\frac{1}{1+m_1}\right)m_1}
{2(1-\frac{1}{1+m_1})} 
= n \frac{m_2-m_1-2m_1^2}{2m_1}.
\]
Since we consider conditional testing given $t$, or equivalently  given $m_1$, we may use the test statistic. 
\beq
\label{scdef}
SB =m_2-m_1-2m_1^2.
\eeq
Actually, we could also have deleted the other terms involving $m_1$. We keep them, however, due to the fact that the sign of $SB$ is of some importance, as explained below. 

Note now that,  since $\theta \ge 0$ is a model restriction, the maximum of (\ref{loggl}) may occur at the boundary point where $\theta=0$. But for $\theta=0$, (\ref{loggl}) is simply the log likelihood of the geometric distribution and is hence maximized by $\pi=\hat p$. Thus if $SB>0$, then we know that the maximum of (\ref{loggl}) is not at a boundary point  with $\theta=0$, and must hence be at a point $(\pi,\theta)$ with $\theta>0$. This point may hence presumably be found by using the partial derivatives derived above. If, instead, $SB<0$, then the maximum likelihood estimate is likely to be at the point $(\hat p,0)$.

It follows from the above that if $SB<0$, then the numerical value is uninteresting, because it corresponds to parameter values outside of the parameter set and intuitively to parameters for which we would not reject the null hypothesis. We therefore suggest to replace $SB$ by $SB_0=\max(0,SB)$ and call this the score test statistic for the null hypothesis $\theta=0$. 

\cite{paul} considered the score test and in addition the likelihood ratio test based on the log likelihood (\ref{loggl}) and standard asymptotics (taking into account the fact that the null hypothesis is on the boundary of the parameter space). He further noted that the likelihood ratio test, as well as the score test, are rather liberal (non-conservative) as regards the size. He therefore found that a bootstrap test might be preferable.

\cite{singh} considered both maximum likelihood estimation and moment estimation of $\alpha$ and $\beta$. The moment estimators are obtained as follows. First, define
\begin{eqnarray*}
	\mu_1 \equiv E(X) &=& \frac{\beta}{\alpha - 1}\; \; \; \;  \mbox{ for } \alpha > 1, \\
	\mu_2 \equiv E(X^2) &=& \frac{\beta(\alpha+2\beta)}{(\alpha - 1)(\alpha-2)}\; \; \; \; \mbox{ for } \alpha > 2.
\end{eqnarray*}
Solving for $\alpha$ and $\beta$ we get
\begin{eqnarray*}
	\alpha &=& \frac{2(\mu_2-\mu_1^2)}{\mu_2-\mu_1-2\mu_1^2}, \\
	\beta &=& \mu_1(\alpha-1)
.
\end{eqnarray*}
The moment estimators $\tilde \alpha$ and $\tilde \beta$ for  $\alpha$ and $\beta$ are obtained by substituting the empirical moments $m_1$ and $m_2$ for $\mu_1$ and $\mu_2$, respectively. This leads to an estimator for the parameter $\theta$ which can be expressed by
\[
\tilde \theta =
(\tilde \alpha + \tilde \beta)^{-1} = 
\frac{m_2-m_1-2m_1^2}{2m_2-m_1^2+m_1m_2}.
\]

It is noticeable that the numerator of $\tilde \theta$ equals $SB$ (see (\ref{scdef})).  Thus $\tilde \theta$ and $SB$ have the same sign (since the denominator above is always positive). As already noted, this sign is of importance for maximum likelihood estimation based on (\ref{loggl}). Note, on the other hand, that in a conditional test using $\hat \theta$  we cannot ignore the denominator of $\tilde \theta$, since it contains $m_2$.

As a final note in this subsection, the test statistic $SB_0$ appears to be essentially identical to the one for the geometric distribution which is derived in  \cite{kyria}. These authors derived  test statistics from  characterizations of distributions given by power-series distribution laws, which include Poisson, binomial, and the negative binomial distribution. Their general goodness-of-fit test statistic is 
\begin{equation}
\label{cm2}
\hat c = \frac{\frac{1}{n} \sum_{i=1}^n X_i(X_i-1)}{\left( \frac{1}{n} \sum_{i=1}^n X_i \right)^2} = \frac{ m_2-m_1}{m_1^2}
\end{equation}
and their test for the geometric distribution rejects the null hypothesis if a normalized version of $$|\hat c - 2| = \frac{|m_2-m_1-2m_1^2|}{m_1^2} $$ is large, where the normalization leads to an asymptotically  standard normal distribution under $H_0$. Since the normalization is a function of $m_1$ only, and the denominator of (\ref{cm2}) can be deleted, using their statistic in a conditional testing we in fact end up with the test statistic $|SB|$.

 \subsubsection*{Example}
 Suppose we have observed the data in Table~\ref{tabex}, which are simulated from the beta-geometric distribution with $n=100, \pi=0.4, \theta=0.125$. 
 
 Using the described test statistics for testing the null hypothesis of a geometric distribution, we obtained the conditional $p$-values given in Table~\ref{pextab}. We note that also the three standard tests are able to detect the deviation from the geometric distribution here, while $p$-values are remarkably lower for the  tests derived above that are tailored for detecting deviations in the direction of a beta-geometric distribution. In fact, the same low $p$-values are obtained for tests versus the discrete Weibull distribution that will be studied below.  
 
 Maximum likelihood estimates for $\pi$ and $\theta$ in the beta-geometric distribution can be calculated using the R-package VGAM, giving   $\hat \pi=0.4274$, $\hat \theta = 0.1166$.   These estimates are used to calculate the estimated expected counts in Table~\ref{tabex}. It is remarkable that these are much closer to the observed values than the ones estimated from the geometric distribution. 
  \begin{table}[H]
      	\begin{center}
      		\begin{tabular}{| c | c | c | c |}
      			\hline
      			$j$ & $o_j$ & $\hat e_j^g$ & $\hat e_j^b$ \\ \hline
      			0 & 42 & 35.4 & 42.7 \\
      			1 & 24 & 22.8 & 21.9 \\
      			2 & 11 & 14.8 &  12.2\\
      			3 & 8  &  9.5 &  7.3 \\
      			4 & 4  &  6.2 & 4.6 \\
      			5 & 4  &   4.0  & 3.0\\
      			6 & 0 &   2.6  &  2.1 \\
      			7 & 1 &  1.7 &  1.4 \\
      			8 & 0  &  1.1  & 1.0 \\
      			9 & 2 &  0.7 &  0.8 \\
      			10 & 2 & 0.4  &  0.6 \\
      			11 & 0 & 0.3  &  0.4 \\
      			12 & 0 & 0.2 &  0.3\\
      			13 & 0 &  0.1  &  0.3 \\
      			14 & 0 & 0.1  &  0.2 \\
      			15 & 1 & 0.0  &  0.2 \\
      			16 & 1 & 0.0 &  0.1\\
      			\hline
      		\end{tabular} 
      	\end{center}
      \caption{Data simulated from the beta-geometric distribution with $n=100, \pi=0.4, \theta=0.125$. The column $o_j$ gives the number of observations $x_i$ that resulted in $x_i=j$. The two last columns give the estimated expected frequencies under a geometric distribution and beta-geometric distribution, respectively.}
      \label{tabex}
      \end{table}
      
      \begin{table}[H]
      	\begin{center}
      		\begin{tabular}{| c | c | c | c | c | c |c|c|c|c|c|}
      			\hline
      			Statistic & $W^2$ & $A^2$ & $KS$ & $CR$ & $SB$ 
      			& $SB_0$ & $\hat \theta$ & $|SW|$ & $SWL$ & $SWU$ \\
      			\hline
      			$p^{\text{cond}}$ & $0.034$ & $0.028$ & $0.059$ & $0.009$ & $0.004$ & 0.004 & 0.004 &0.005 &0.004 &0.996  \\
      			\hline
      		\end{tabular} 
      	\end{center}
      \caption{Conditional $p$-values obtained by simulating 10000 data sets from the conditional distribution}
      \label{pextab}
      \end{table}

\subsubsection{The discrete Weibull distribution of type I}\label{Weibulldisc}
Let for $x=0,1,2,\ldots$,
\begin{equation}
\label{typeI}
  \mathbf{P}(X=x) =  q^{x^\beta}-q^{(x+1)^\beta}
  \end{equation}
  where $0<q<1$ and $\beta>0$. 
 This is the probability distribution of the   type I Weibull distribution, which was introduced by \cite{naka}. We denote it by $\mathcal{W}(q,\beta)$. The geometric distribution with parameter $p$ is now a special case obtained when $q = 1-p$ and $\beta=1$. 
 
 The  R-package DiscreteWeibull contains routines for this distribution, including simulation of data and estimation of parameters. 
 
 The discrete hazard rate of a random variable with values in the (nonnegative) integers can be defined by \citep{barlow63} $\lambda(x) = \mathbf{P}(X=x\; |\; X \ge x)$. From (\ref{typeI}) we get 
 \[
   \lambda(x) = \frac{\mathbf{P}(X=x)}{\mathbf{P}(X \ge x)} = \frac{q^{x_i^\beta}-q^{(x_i+1)^\beta}}{q^{x_i^\beta}}= 1-q^{(x+1)^\beta-x^\beta}
 \]
 which is seen to be increasing in $x$ if $\beta>1$ and decreasing in $x$ if $\beta < 1$, and constant equal to $p$ when $\beta=1$, which corresponds to the geometric distribution. 
 
 Suppose now we have data $x_1,x_2,\ldots,x_n$ with values in $\mathbb{N}_0$. Testing the null hypothesis that the data come from the geometric distribution, is now equivalent to testing $H_0: \beta=1$ vs. $H_1: \beta \neq 1$, or possibly the one-sided versions of the alternative. The testing  can be done by a likelihood ratio test.  
 It follows  from (\ref{typeI}) that the log-likelihood for the sample $x_1,x_2,\ldots,x_n$ from the type I Weibull distribution is given by
 \[
    \ell (q,\beta) = \sum_{i=1}^n \ln\left( q^{x_i^\beta}-q^{(x_i+1)^\beta} \right).
 \]
 The likelihood ratio test statistic can be computed by calculating the maximum likelihood estimates of $q$ and $\beta$, and of $p$, which is the parameter under the null hypothesis model. Details are given by \cite{vila2019}, while computations can be done using the R-package DiscreteWeibull.
 
A score test can be derived in a way similar to what we did in Section~\ref{betageo} for the beta-geometric distribution.  First, the partial derivative with respect to $\beta$ of the log-likelihood function $\ell$,  is given by
\[
    \frac{\partial \ell}{\partial \beta} =\sum_{i=1}^n 
    \frac{q^{x_i^\beta}x_i^\beta \ln( q) \ln (x_i)
    - q^{(x_i+1)^\beta}(x_i+1)^\beta \ln( q) \ln (x_i+1)
    }{q^{x_i^\beta}-q^{(x_i+1)^\beta}}
 \]
 The score statistic of $H_0$ can then be found by letting $\beta=1$, which leads to
 \begin{equation}
 \label{scw}
 \frac{\partial \ell}{\partial \beta}|_{\beta=1}=  
   \frac{\ln(q)}{1-q} \sum_{i=1}^n \left(x_i \ln (x_i)
    - q(x_i+1) \ln (x_i+1) \right).
    \end{equation}
   It can now be checked that if the $x_i$ are from the geometric distribution with parameter $p$, the expected value of (\ref{scw}) is 0 (noting that $q=1-p$). The standard approach  is now to estimate $q$ by $1-\hat p$ (from the geometric distribution) and divide (\ref{scw}) by the estimated standard deviation, in order to obtain a test statistic which is standard normally distributed under the null hypothesis. We shall, however, consider conditional testing, conditioning on   $\sum_{i=1}^nX_i$ or, equivalently, on $\hat p$, and we may hence use the test statistic
   \[
   SW =  \sum_{i=1}^n \left[ (1-\hat p)(x_i+1) \ln (x_i+1) - x_i \ln (x_i)
     \right] .
   \]
   where $\hat p = n/(n+\sum_ix_i)$.
 If  $x_i=0$, we shall let $x_i \ln(x_i) = 0$. Note that we have changed the order of the terms inside the sum as compared to (\ref{scw}). This is because $\ln(q) < 0$ and will lead to a statistic $SW$ with the same sign as $\frac{\partial \ell}{\partial \beta}|_{\beta=1}$. Then for the two-sided alternative, $\beta \neq 1$, we should use the statistic $|SW|$ as the test statistic. A more powerful test can then be defined for the two one-sided alternatives, by using  $SWU = SW$ if the alternative is $\beta>1$ and $SWL=-SW$ for the alternative $\beta <1$, and reject in both cases for high values of the test statistic. 
 
 The resemblance between the statistics $SW$ and $CR$ is striking. In fact, $CR$ is obtained from $SW$ by letting $\hat p=0$ and switching the sign.  Simulations and $p$-value calculations in the following will indicate the possible difference between their merits. 

\subsubsection*{Example}
 Suppose we have observed the data in Table~\ref{tabex2}, which are simulated from a type I discrete Weibull distribution with $n=50, q=0.8, \beta=1.4$ using the R-package DiscreteWeibull.
 
 Using all the tests considered so far in the paper, we obtained the conditional $p$-values given in Table~\ref{pextab2}. We note that also the three standard tests are able to detect the deviation from the discrete Weibull distribution here, while $p$-values are remarkably lower for the tests $|SW|$ and $SWL$, which are tailored for detecting deviations in the direction of the discrete Weibull distribution.  It should be noted, however, that the test based on   $CR$ as well as the tests versus the beta-geometric distribution are useless for these data. The reason for this last fact is that the beta-geometric distribution always increases the variance of the data as compared to the geometric distribution, while the discrete Weibull with $\beta > 1$ decreases the variance (a property well known for the continuous Weibull distribution). Thus a beta-geometric distribution would have difficulties fitting these data.  
 
 Maximum likelihood estimates for $\pi$ and $\theta$ are calculated as $\hat q = 0.7239$, $\hat \beta = 1.267$ using the R-package DiscreteWeibull.   These estimates are used to calculate the estimated expected counts in Table~\ref{tabex2}. Again, these are much closer to the observed values than the ones estimated from the geometric distribution.  
  \begin{table}[H]
      	\begin{center}
      		\begin{tabular}{| c | c | c | c |}
      			\hline
      			$j$ & $o_j$ & $\hat e_j^g$ & $\hat e_j^b$ \\ \hline
      			0 & 13 & 18.0 & 13.8 \\
      			1 & 14 & 11.5 & 13.2 \\
      			2 & 10 & 7.4 &  9.3\\
      			3 & 8  &  4.7 &  5.9 \\
      			4 & 1  &  3.0 & 3.5 \\
      			5 & 1  &   1.9  & 2.0\\
      			6 & 0 &   1.2  &  1.1 \\
      			7 & 2 &  0.8 &  0.6 \\
      			8 & 1  &  0.5  & 0.3 \\
      			\hline
      		\end{tabular} 
      	\end{center}
      \caption{Data simulated from a type I Weibull distribution with $n=50, q=0.8, \beta=1.4$. The column $o_j$ gives the number of observations $x_i$ that resulted in $x_i=j$. The two last columns give the estimated expected frequencies under a geometric distribution and type I Weibull distribution, respectively.}
      \label{tabex2}
      \end{table}
      
      \begin{table}[H]
      	\begin{center}
      			\begin{tabular}{| c | c | c | c | c | c |c|c|c|c|c|}
      			\hline
      			Statistic & $W^2$ & $A^2$ & $KS$ & $CR$ & $SB$ 
      			& $SB_0$ & $\hat \theta$ & $|SW|$ & $SWL$ & $SWU$ \\
      			\hline
      			$p^{\text{cond}}$ & $0.072$ & $0.078$ & $0.124$ & $0.962$ & $0.890$ & 1.0 & 0.890 &0.083 &  0.956 & 0.044 \\
      			\hline
      		\end{tabular} 
      	\end{center}
      \caption{Conditional $p$-values obtained by simulating 10000 data sets from the conditional distribution. Two-sided (one-sided) testing using $|SW|$ ($SWU$)  means that the alternative hypothesis is $\beta \neq 1$ ($\beta > 1$). The statistics $SB$ and $SB_0$  can only be used to detect an increased variance compared to the geometric distribution. Here the test statistic $SB$ is negative, which makes these tests meaningless. }
      \label{pextab2}
      \end{table}

\section{Computer Simulations}
\subsection{Power Study}
We did a power study with various sample sizes, alternative distributions and all previously defined test statistics with significance level $\alpha=0.1$. In some cases we disregarded sample sizes $n=5$ or $n=100$ if the powers were too close to the significance level or $1$. We used $1000$ iterations to calculate each conditional $p$-value and another $1000$ iterations to calculate the power. These numbers were chosen to make sure each power calculation takes less than half an hour of computation time.
\begin{table}[H]
\begin{center}
	\resizebox{\columnwidth}{!}{%
	\begin{tabular}{| c | c | c | c | c | c | c | c | c | c |}
		\hline 
		Alternative & Sample size & $W^2$ & $A^2$ & $KS$ & $CR$ & $SB_0$  & $|SW|$ & $SWL$ & $SWU$\\
		\hline
		\multirow{2}{*}{$\text{Pois}(0.5)$}
		 & $n=25$ & $0.231$ & $0.238$ & $0.208$ & $0.002$ & $0.001$ & $0.225$ & $0.002$ & $0.329$ \\
		 & $n=100$ & $0.736$ &  $0.734$ & $0.705$ & $0.000$ & $0.000$ & $0.763$ & $0.000$ & $0.851$ \\
		\hline
		\multirow{3}{*}{$\text{Pois}(1)$} & $n=5$ & $0.119$ & $0.119$ & $0.110$ & $0.008$ &  $0.008$ & $0.122$ & $0.008$ & $0.134$ \\
		 & $n=25$ & $0.613$ & $0.605$ & $0.543$ & $0.000$ & $0.001$ & $0.618$ & $0.000$ & $0.730$ \\
		 & $n=100$ & $0.996$ & $0.996$ & $0.992$ & $0.000$ & $0.000$ & $0.998$ & $0.000$ & $0.999$ \\
		\hline
		\multirow{2}{*}{$\text{Pois}(2)$}
		 & $n=5$ & $0.332$ & $0.321$ & $0.163$ & $0.003$ & $0.003$ & $0.339$ & $0.003$ & $0.395$ \\
		 & $n=25$  & $0.963$ & $0.965$ & $0.914$ & $0.000$ & $0.000$ & $0.966$ & $0.000$ & $0.985$\\
		\hline
		\multirow{2}{*}{$\text{Bin}(5,0.3)$} & $n=5$ & $0.418$ & $0.410$ & $0.294$ & $0.000$ & $0.000$ & $0.403$ & $0.000$ & $0.432$ \\
		 & $n=25$ & $0.986$ & $0.985$ & $0.972$ & $0.000$ & $0.000$ & $0.990$ & $0.000$ & $0.996$ \\
		\hline
		\multirow{2}{*}{$\text{NB}(5, 0.5)$} & $n=5$ & $0.531$ & $0.466$ & $0.434$ & $0.000$ &  $0.001$ & $0.538$ & $0.000$ & $0.672$ \\
		 & $n=25$ & $0.997$ & $0.998$ & $0.986$ & $0.000$ & $0.000$ & $1.000$ & $0.000$ & $1.000$ \\
		\hline
		\multirow{2}{*}{$\text{NB}(3, 0.7)$} & $n=25$ & $0.395$ & $0.393$ & $0.339$ & $0.001$ &  $0.002$ & $0.407$ & $0.001$ & $0.526$ \\
		 & $n=100$ & $0.875$ & $0.875$ & $0.834$ & $0.000$ & $0.000$ & $0.906$ & $0.000$ & $0.945$ \\
		\hline
		\multirow{2}{*}{$\mathrm{BG}(2, 5)$} & $n=5$ & $0.161$ & $0.170$ & $0.137$ & $0.267$ & $0.269$ & $0.158$ & $0.274$ & $0.025$ \\
		& $n=100$ & $0.565$ & $0.565$ & $0.514$ & $0.676$ & $0.705$ & $0.608$ & $0.706$ & $0.007$ \\
		\hline
		\multirow{2}{*}{$\mathrm{BG}(2, 2)$} & $n=5$ & $0.122$ & $0.132$ & $0.108$ & $0.205$ &  $0.201$ & $0.125$ & $0.207$ & $0.018$ \\
		& $n=25$ & $0.558$ & $0.570$ & $0.504$ & $0.705$ & $0.688$ & $0.611$ & $0.717$ & $0.001$ \\
		\hline
		\multirow{2}{*}{$\mathcal{W}(0.7, 0.8)$} & $n=25$ & $0.320$ & $0.338$ & $0.282$ & $0.503$ & $0.432$ & $0.353$ & $0.492$ & $0.006$ \\
		& $n=100$ & $0.749$ & $0.759$ & $0.679$ & $0.874$ & $0.792$ & $0.792$ & $0.882$ & $0.000$ \\
		\hline
		\multirow{2}{*}{$\mathcal{W}(0.5, 1.5)$} & $n=25$ & $0.428$ & $0.431$ & $0.385$ & $0.000$ & $0.001$ & $0.438$ & $0.000$ & $0.567$ \\
		& $n=100$ & $0.948$ & $0.943$ & $0.937$ & $0.000$ & $0.000$ & $0.960$ & $0.000$ & $0.984$ \\
		\hline
		\end{tabular}
	}
\end{center}
\caption{Conditional power calculations with significance level $\alpha=0.1$.}
\end{table}
As usual for a wide choice of alternative distributions, there is no best test against all alternatives. From standard tests, $W^2$ has slightly higher powers with small sample sizes. For larger sample sizes, $A^2$ and $W^2$ are almost identical. Maximal type test $KS$ has slightly lower powers than the other standard tests. Tests $CR$, $SB$, $SB_0$, $\hat\theta$, $SWL$ and $SWU$ are sensitive to the alternative distribution the data comes from. This makes them situational and they lack the versatility of the standard tests. For example, $CR$, $SB$, $SB_0$, $\hat\theta$ and $SWL$ outperform the standard tests when the data comes from $\mathrm{BG}$ or $\mathcal{W}$ distributions. $SWU$ outperforms other tests for $\text{Pois}$, $\text{Bin}$ and $\text{NB}$ distributions. $|SW|$ is more versatile and has almost identical powers to the standard quadratic tests. Likelihood based tests need a versatile comparative alternative distribution to perform well. Type I Weibull distribution fits this role, as we can see from $|SW|$ powers. Test $SWL$ is for the case where $\beta<1$ and $SWU$ for $\beta>1$. Under those conditions, they outperform $|SW|$.

\subsection{Type I Errors}
The type I error is defined to be the probability of falsely rejecting the null hypothesis if it is actually true. We simulated this scenario by drawing samples under the null hypothesis, from the geometric distribution with parameter $p$ for various sample sizes $n$. If the conditional $p$-value came out lower than the significance level, we had made a type I error.

We used $1000$ iterations to calculate each conditional $p$-value and another $1000$ iterations to calculate the type I error.
\begin{table}[H]
	\begin{center}
		\resizebox{\columnwidth}{!}{%
			\begin{tabular}{| c | c | c | c | c | c | c | c | c | c |}
				\hline
				\multicolumn{10}{|c|}{$\alpha=0.05$}\\
				\hline 
				Parameter & Sample size & $W^2$ & $A^2$ & $KS$ & $CR$ & $SB_0$  & $|SW|$ & $SWL$ & $SWU$\\
				\hline
				\multirow{3}{*}{$p=0.25$} & $n=5$ & $0.048$ & $0.047$ & $0.035$ & $0.037$ & $0.035$ & $0.044$ & $0.037$ & $0.039$\\
				& $n=25$ & $0.055$ & $0.054$ & $0.047$ & $0.048$ & $0.052$ & $0.052$ & $0.049$ & $0.050$ \\
				& $n=100$ & $0.056$ & $0.056$ & $0.055$ & $0.054$ & $0.056$ & $0.057$ & $0.056$ & $0.053$ \\
				\hline
				\multirow{3}{*}{$p=0.5$} & $n=5$ & $0.027$ & $0.026$ & $0.018$ & $0.018$ & $0.017$ & $0.023$ & $0.017$ & $0.020$\\
				& $n=25$ & $0.066$ & $0.065$ & $0.056$ & $0.050$ & $0.048$ & $0.064$ & $0.051$ & $0.055$ \\
				& $n=100$ & $0.042$ & $0.045$ & $0.036$ & $0.050$ & $0.049$ & $0.039$ & $0.059$ & $0.050$ \\
				\hline
				\multirow{3}{*}{$p=0.75$} & $n=5$ & $0.002$ & $0.002$ & $0.000$ & $0.001$ & $0.001$ & $0.001$ & $0.001$ & $0.001$\\
				& $n=25$ & $0.032$ & $0.030$ & $0.027$ & $0.032$ & $0.027$ & $0.029$ & $0.031$ & $0.017$ \\
				& $n=100$ & $0.045$ & $0.045$ & $0.040$ & $0.046$ & $0.040$ & $0.046$ & $0.041$ & $0.043$ \\
				\hline
				\multicolumn{10}{|c|}{$\alpha=0.1$}\\
				\hline
					\multirow{3}{*}{$p=0.25$} & $n=5$ & $0.070$ & $0.066$ & $0.057$ & $0.080$ & $0.081$ & $0.073$ & $0.079$ & $0.073$\\
					& $n=25$ & $0.105$ & $0.113$ & $0.096$ & $0.111$ & $0.118$ & $0.103$ & $0.112$ & $0.106$ \\
					& $n=100$ & $0.101$ & $0.092$ & $0.087$ & $0.104$ & $0.101$ & $0.103$ & $0.097$ & $0.107$ \\
					\hline
					\multirow{3}{*}{$p=0.5$} & $n=5$ & $0.042$ & $0.040$ & $0.038$ & $0.035$ & $0.037$ & $0.039$ & $0.035$ & $0.038$\\
					& $n=25$ & $0.101$ & $0.108$ & $0.071$ & $0.095$ & $0.089$ & $0.110$ & $0.098$ & $0.107$ \\
					& $n=100$ & $0.123$ & $0.116$ & $0.117$ & $0.105$ & $0.104$ & $0.109$ & $0.105$ & $0.115$ \\
					\hline
					\multirow{3}{*}{$p=0.75$} & $n=5$ & $0.006$ & $0.006$ & $0.005$ & $0.002$ & $0.002$ & $0.006$ & $0.002$ & $0.006$\\
					& $n=25$ & $0.062$ & $0.060$ & $0.060$ & $0.059$ & $0.057$ & $0.070$ & $0.057$ & $0.032$ \\
					& $n=100$ & $0.101$ & $0.102$ & $0.090$ & $0.097$ & $0.095$ & $0.095$ & $0.098$ & $0.078$ \\
					\hline
			\end{tabular}
		}
	\end{center}
	\caption{Type I errors for various sample sizes, parameter $p$-values, test statistics and significance levels $\alpha$.}
\end{table}
Some of the type I errors are slightly above the significance level but this is explained by Monte Carlo errors from calculating the $p$-value and the error. This is because of the discreteness of the data. If the parameter is $p=0.75$, the samples consist largely of $0$-s and if the sample size is small, we often get only $0$-s. In that case $t=0$ and it is a singular case.

We left out $\hat\theta$ and $SB$ from the power study and type I error study because they had identical powers for all alternatives and $SB_0$ should be preferred over $SB$.

\section{Real Life Data}
In this section we use real life data from \cite{BCG}. The data consist of numbers of inspections between discovery of defects in an industrial process. Conditional samples are used to calculate the distribution of goodness-of-fit test statistics following the recipe from Section~\ref{MC}. Conditional $p$-values are reported to decide if we should reject or not reject the null hypothesis, that the data comes from the geometric distribution.

In order to have the data on the form considered in this paper, we have subtracted 1 from each observation. 

 \begin{table}[H]
\begin{center}
\begin{tabular}{| c | c | c | c | c | c | c | c | c | c | c | c | c | c | c |}
	\hline
	Value & $0$ & $1$ & $2$ & $3$ & $4$ & $\ge 5$ \\
	\hline
	Observed frequency & $6$ & $4$ & $3$ & $3$ & $2$ & $10$  \\
	Expected frequency, geometric & 3.9 & 3.3 & 2.9 & 2.5 & 2.1 & 13.3\\
	Expected freqquency, beta-geometric & 5.0 & 3.9 & 3.1 & 2.5 & 2.0 & 11.5\\
	Expected frequency, discrete Weibull & 6.0 & 3.6 & 2.7 & 2.2 & 1.8 & 11.7\\
	\hline
	\end{tabular} 
\end{center}
\caption{Real life data. Observed and estimated expected frequencies for three different models.}
      \label{realdata}
      \end{table}
      The data is given in Table~\ref{realdata}. Note that the data $x_i \ge 5$ are lumped together in the table for illustrative purposes. The observed values for these 10 observations are 6. 8, 10, 12, 13, 16, 17, 25, 28, and are used in the simulations and calculations. 
      
     Conditional $p$-values for the various tests are given in Table~\ref{realtab}, calculated with $10000$ Monte Carlo samples from the conditional distribution.
       \begin{table}[H]
      	\begin{center}
      		\begin{tabular}{| c | c | c | c | c | c |c|c|c|c|c|}
      			\hline
      			Statistic & $W^2$ & $A^2$ & $KS$ & $CR$ & $SB$ 
      			& $SB_0$ & $\hat \theta$ & $|SW|$ & $SWL$ & $SWU$ \\
      			\hline
      			$p^{\text{cond}}$ & $0.107$ & $0.117$ & $0.315$ & $0.042$ & $0.134$ & 0.134 & 0.134 &0.110 & 0.047 &0.953  \\
      			\hline
      		\end{tabular} 
      	\end{center}
      \caption{Conditional $p$-values obtained by simulating 10000 data sets from the conditional distribution.}
      \label{realtab}
      \end{table}
      
The standard tests as well as the tests versus beta-geometric distribution still indicate the possibility of a geometric distribution, having $p$-values $> 0.10$, while the hypothesis of geometric distribution is in fact rejected at 5\% significance level by the $CR$ test and the one-sided test versus the type I discrete Weibull distribution with $\beta < 1$.  This possibility of the Weibull distribution is also indicated by the fitted expected frequencies as shown in Table~\ref{realdata}. 

The VGLM R-package gives the maximum likelihood estimates for the beta-geometric model given by
\begin{eqnarray*}
\hat \pi &=& 0.1772 \\
\hat \theta &=&  0.0502 \\
\hat p &=& 0.1378 \mbox{ (geometric distribution)}
\end{eqnarray*}

Also, the VGLM R-package gives a $p$-value for a likelihood ratio test verus the beta-geometric  to be 0.3276. This is higher than the values for $SB$ and $SB_0$, e.g. The reason might be that the asymptotic chi-square distribution of the likelihood ratio is not appropriate for these data. 

The DiscreteWeibull R-package estimates a type I Weibull model giving
$\hat q =0.784, \hat \beta= 0.794$, indicating a decreasing hazard rate, which corresponds well to the above rejection of the one-sided test versus $\beta<1$.

\section{Conclusion and Future Work}
In this paper we studied goodness-of-fit tests for discrete distributions obtained by conditioning on the sufficient statistic under the null hypothesis. We developed in particular a method to draw conditional samples from the geometric distribution. These samples are used for calculation of $p$-values for
various goodness-of-fit tests. 
In addition to considering standard goodness-of-fit tests,  we derived new likelihood based test statistics for testing of the geometric distribution versus heterogeneity, as well as
versus discrete Weibull distributions with both increasing and decreasing hazard.

A power study was conducted to check how the tests perform against data from different alternative distributions. Our simulations suggested that the two-sided test versus the type I discrete Weibull distributions was able to detect bad fit for data from various alternative distributions. The power results for this test, $|SW|$, were in fact generally similar to the ones obtained for the standard quadratic goodness-of-fit tests.

The tests versus heterogeneous geometric distributions, $CR$ and $SB_0$, are doing well for alternatives of this kind, as one should expect, and then usually much better than the standard tests. The tests for heterogeneity are, however, mostly inferior versus other alternatives. The reason is presumably that heterogeneity leads to increased variance as compared to the geometric distribution. On the other hand, Weibull distributions with decreasing hazard lead to an increased variance.  

Real life data from \cite{BCG} was considered and it was tested whether the geometric
distribution is a good fit. The calculated $p$-values suggested that the geometric distribution might not be a good fit according to some of the tests.

For further work, a general method could be described for the case where $T(\bX) = \sum_{i=1}^n  X_i$ is sufficient for the family of distributions under the null hypothesis. This is the case for the power series distributions where the probability distribution is of the form, see
\cite{kyria} or \cite{gonzalez},
\begin{equation}
\label{powdist}
\mathbf{P}(X=x) = \frac{a(x)\theta^x}{\eta(\theta)} \mbox{ for } x=0,1,2,\ldots
\end{equation}
where $a(x) \ge 0$, $\theta>0$, $\eta(\theta)=\sum_{y=0}^\infty a(y)\theta^y$. 
The Poisson, binomial, negative binomial and geometric distributions are all of this kind. It can be shown from (\ref{powdist}) \citep{gonzalez} that for  samples $\bX=(X_1,X_2,\ldots,X_n)$ from this distribution, we have
\[
\mathbf{P}(\bX=\bx\; |\; T(\bX)=t) \propto \prod_{i=1}^n a(x_i) \mbox{ when } \sum_{i=1}^n x_i= t.
\] 
Conditional samples with a given sum $t$ can hence be obtained by the Metropolis-Hastings algorithm using samples from the conditional geometric distribution as proposals.

\newpage
\section*{Appendix}

	\noindent \textbf{Proof of Lemma~\ref{biject}}
\begin{proof}
	The function is defined correctly, because for any $(k_1,k_2,\ldots, k_{n-1})\in L_2$
	$$k_1-1 + (k_2-2)-(k_1-1) + \ldots +(k_{n-1}-(n-1))-(k_{n-2}-(n-2)) + t-(k_{n-1}-(n-1)) = t $$
	and $k_1-1, (k_2-2)-(k_1-1), \ldots ,t-(k_{n-1}-(n-1)) \in \mathbb{N}_0$, as $0<k_1<k_2<\ldots <k_{n-1}<t+n$. 
	\\
	Let us assume that $(k_1,k_2,\ldots ,k_{n-1}),(v_1,v_2,\ldots ,v_{n-1})\in L_2$, such that
	$$\phi(k_1,k_2,\ldots ,k_{n-1})=\phi(v_1,v_2,\ldots,v_{n-1}).$$
	This implies that
	\begin{align*}
	k_1+1=v_1+1&\Rightarrow k_1=v_1,\\
	(k_2-2)-(k_1-1)=(v_2-2)-(v_1-1)&\Rightarrow k_2=v_2,\\
	&\ldots\\
	(k_{n-1}-(n-1))-(k_{n-2}-(n-2))=(v_{n-1}-(n-1))-(v_{n-2}-(n-2))&\Rightarrow k_{n-1}=v_{n-1},\\
	\end{align*}
	and we can conclude that $\phi$ is injective.
	\\	
	Let us fix an element $(x_1,x_2,\ldots ,x_n)\in L_1$, then 
	$$\phi\left(x_1+1,x_1+x_2+2, \ldots , \sum_{i=1}^{n-1}x_i+(n-1)\right)=(x_1,x_2,\ldots ,x_n)$$
	and
	$$0<x_1+1<x_1+x_2+2<\ldots <\sum_{i=1}^{n-1}x_i+n-1<t+n,$$
	which implies that $\phi$ is surjective. Injectivity and surjectivity imply that $\phi$ is bijective.
	\end{proof}

\begin{lemma}
\label{correct}
	Let $n,t\in\mathbb{N}$ and $(k_1,\ldots ,k_{n-1})\in L_2$ be an arbitrary sample drawn according to algorithm \ref{norec}. Then it is drawn uniformly, i.e. 
	$$\mathbf{P}(k_1,\ldots ,k_{n-1})=\frac{1}{\binom{n+t-1}{n-1}}$$
	for each $(k_1,\ldots ,k_{n-1})\in L_2$. The probability follows from Lemma \ref{many}.
	\end{lemma}
\begin{proof}
The task is to distribute $t$ stars and $n-1$ bars randomly on the positions $1,2,\ldots,t+n-1$. We start from the right. Then the probability of placing a bar in position $t+n-1$ is $\frac{n-1}{t+n-1}$. Then we proceed conditionally to the left and multiply probabilities of placing bars or stars in order to calculate probabilities of a given configuration. 

More precisely, let $(k_1,k_2,\ldots ,k_{n-1})$ be an arbitrary sample drawn according to Algorithm \ref{norec}. We want to calculate
	\begin{equation}\label{p}\mathbf{P}(k_1,\ldots ,k_{n-1}).\end{equation}
	We know that the algorithm accepted $n-1$ integers (i.e., placements $k_j$ of the bars) in the process. Also, let $V$ denote the number of integers that were not accepted. In total the algorithm ran $V+n-1$ iterations. The probability \eqref{many} is a product of $V+n-1$  probabilities. Let us look at the denominator and nominator separately. In the denominator we have
	\begin{equation}\label{den}(t+n-1)\cdot(t+n-2)\cdot \ldots \cdot(t-V-1).\end{equation}
	In the numerator we have
	\begin{equation}\label{nom1}(n-1)\cdot(n-2)\cdot \ldots \cdot 2 \cdot 1=(n-1)!\end{equation}
	from the accepted integers. In the numerator there is also
	\begin{equation}\label{nom2}t\cdot(t-1)\cdot \ldots \cdot (t-V-1)\end{equation}
	from the integers that were not accepted. Combining \eqref{den}, \eqref{nom1} and \eqref{nom2} we get
	\begin{align*}
	\mathbf{P}(k_1,k_2,\ldots ,k_{n-1})&=\frac{(n-1)!t(t-1)\cdots (t-V-1)}{(t+n-1)\cdots (t-V-1)}\\
	&=\frac{(n-1)!t(t-1)\cdots (t-V-1)(t-V-2)\cdots 2\cdot 1}{(t+n-1)\cdots (t-V-1)(t-V-2)\cdots 2\cdot 1}\\
	&=\frac{(n-1)!\; t!}{(t+n-1)!}\\
	&=\frac{1}{\binom{t+n-1}{n-1}}.
	\end{align*}

	\end{proof}

\begin{algorithm}[H]
	\KwData{$t$ and $n$}
	\KwResult{$k_1,\ldots ,k_{n-1}$}
	initialization\;
	$N=0$ \tcp*{number of accepted integers} 
	$V=0$ \tcp*{number of not accepted integers} 
	$I=t+n-1$ \tcp*{integer to consider} 
	\While{$N<n-1$}{
		Draw $p \sim U[0,1]$\;
		\If{$p<(n-1-N)/(t+n-1-N-V)$}{
			$k_{n-1-N}=I$ \tcp*{integer $I$ was accepted} 
			$N=N+1$\;
			$I=I-1$\;
			\textbf{Continue} 
		}
		\If{$p\geq(n-1-N)/(t+n-1-N-V)$}{
			$V=V+1$ \tcp*{integer $I$ was not accepted}
			$I=I-1$\;
			\textbf{Continue} 
		}
	}
	\caption{Draw $k_1,\ldots ,k_{n-1}$ Uniformly}
	\label{norec}
\end{algorithm}
\newpage
\begin{algorithm}[H]
	\KwData{Integer $V$, Data set $x$, sufficient statistic $T(x)=t$ and test statistic $D$}
	\KwResult{$p^{\text{cond}}$}
	initialization\;
	$\text{count} = 0$\;
	\For{$1$ to $V$}{
		Draw $y \sim X\; |\; T(X)=t$ \;
		\If{$D(y) \ge D(x)$}{
			$\text{count} = \text{count}+1$\;
		}
		}
	$p^{\text{cond}}=\text{count}/V\;$
	\caption{Monte Carlo Conditional $p$-value}
	\label{condp}
\end{algorithm}

\vekk{ 
\subsection{Sample from the Conditional Binomial Distribution}
Let $\mathbf{X}=(X_1, X_2, \ldots ,X_n)$ be a vector of iid random variables, such that  $X_i \sim \text{B}(m_i,p)$ for all $i=1,2,\ldots ,n$ and $T(\mathbf{X})=\sum_{i=1}^nX_i=t$ is a sufficient statistic for $p$. Conditional distribution is calculated as 
\begin{align*}
\mathbf{P}(\mathbf{X}=\mathbf{x}\; |\; T(\mathbf{X})=t)&= \mathbf{P}(X_1=x_1, X_2=x_2, \ldots ,X_n=x_n\; | \; T(X_1,X_2,\ldots,X_n)=t)\\
&=\frac{\mathbf{P}(X_1=x_1,X_2=x_2,\ldots ,X_n=x_n, \sum_{i=1}^nX_i=t)}{\mathbf{P}(\sum_{i=1}^nX_i=t)}\\
&=\left\{\begin{array}{lr}\displaystyle{\frac{\mathbf{P}(X_1=x_1,X_2=x_2,\ldots ,X_n=x_n)}{\mathbf{P}(\sum_{i=1}^nX_i=t)}}, & \text{ if } \sum_{i=1}^nX_i=t\\ 0, & \text{ if } \sum_{i=1}^nX_i\neq t\end{array}\right. .
\end{align*}
Now, we restrict the support to be $\{ (x_1,x_2,\ldots ,x_n)\; \colon \; \sum_{i=1}^nx_i=t, \ x_1,x_2,\ldots ,x_n \in \mathbb{N}_0\}$. Then
\begin{align*}
\mathbf{P}(\mathbf{X}=\mathbf{x}\; | \; T(\mathbf{X})=t)&=\frac{\mathbf{P}(X_1=x_1,X_2=x_2,\ldots ,X_n=x_n)}{\mathbf{P}(\sum_{i=1}^nX_i=t)}\\
&=\frac{\binom{m_1}{x_1}p^{x_1}(1-p)^{m_1-x_1} \binom{m_2}{x_2}p^{x_2}(1-p)^{m_2-x_2}\cdots \binom{m_n}{x_n}p^{x_n}(1-p)^{m_n-x_n}}{\binom{m_1+m_2+\ldots +m_n}{t}p^t(1-p)^{m_1+m_2+\ldots +m_n-t}}\\
&= \frac{\binom{m_1}{x_1}\binom{m_2}{x_2}\ldots \binom{m_n}{x_n}}{\binom{m_1+m_2+\ldots +m_n}{t}}.
\end{align*}
Which means that the conditional distribution has the multivariate hypergeometric distribution. \cite{randomnumber} provides a method for drawing realizations from it.

\subsection{Conditional Poisson Distribution}

Let $\mathbf{X}=(X_1, X_2, \ldots ,X_n)$ be a vector of iid random variables, such that  $X_i \sim \text{Pois}(a_i\lambda)$ for all $i=1,2,\ldots ,n$ and $a_1,a_2,\ldots ,a_n>0$ are known. Function $T(\mathbf{X})=\sum_{i=1}^nX_i=t$ is a sufficient statistic for $\lambda$. Sufficiency follows from the joint probability factorization
$$\mathbf{P}(\mathbf{X}=\mathbf{x})=\prod_{i=1}^n \frac{a_i^{x_i} \lambda^{x_i}e^{-a_i\lambda}}{x_i!}=e^{-\sum_{i=1}^na_i\lambda}\lambda^{\sum_{i=1}^nx_i}\prod_{i=1}^n \frac{a_i^{x_i} }{x_i!}.$$
Just as in the previous case, we restrict the support to be $\{ (x_1,x_2,\ldots ,x_n)\; \colon \; \sum_{i=1}^nx_i=t, \ x_1,x_2,\ldots ,x_n\in \mathbb{N}_0 \}$. Then
\begin{align}
\mathbf{P}(\mathbf{X}=\mathbf{x} \; |\; T(\mathbf{X})=t)&= \mathbf{P}(X_1=x_1, X_2=x_2, \ldots ,X_n=x_n\; | \; T(X_1,X_2,\ldots,X_n)=t)\nonumber\\
&=\frac{\prod_{j=1}^n \frac{(a_j\lambda)^{x_j}e^{-a_j\lambda}}{x_j!}}{\mathbf{P}\left( \sum_{i=1}^nX_i=t\right)}\nonumber\\
&=\frac{\prod_{j=1}^n \frac{(a_j\lambda)^{x_j}e^{-a_j\lambda}}{x_j!}}{\frac{e^{-\lambda\sum_{i=1}^na_i}\left(\sum_{i=1}^na_i\lambda\right)^t}{t!}}\label{condpois}\\
&=\frac{t!\lambda^{\sum_{j=1}^nx_j} }{(\sum_{i=1}^na_i\lambda)^t \prod_{i=1}^n x_i!} \prod_{i=1}^n a_i^{x_i}\nonumber \\
&=\frac{t!}{ \prod_{i=1}^n x_i!} \prod_{i=1}^n \left(\frac{a_i}{\sum_{i=1}^na_i}\right)^{x_i} \label{mult}
\end{align}
and \eqref{condpois} comes from $\sum_{i=1}^nX_i\sim \text{Pois}\left(\lambda\sum_{i=1}^na_i\right)$. Equation \eqref{mult} shows that the sample is multinomially distributed, i.e.
$$\mathbf{X}\; | \; T(\mathbf{X})=t \sim \text{mult}\left(t, \frac{a_1}{\sum_{i=1}^na_i},\frac{a_2}{\sum_{i=1}^na_i},\ldots , \frac{a_n}{\sum_{i=1}^na_i}\right).$$
\cite{randomnumber} provides a method to generate a sample from the multinomial distribution.
} 

\subsection*{Conditional Sampling from the Negative Binomial Distribution}

Let $Y_1\sim \text{NB}(r_1,p), Y_2\sim \text{NB}(r_2,p),\ldots,Y_n\sim \text{NB}(r_n,p) $ be independent random variables, where the parameters $r_1,\ldots ,r_n$ are assumed to be known. Then $T(\mathbf{Y})=\sum_{i=1}^nY_i$ is a sufficient statistic for $p$. In this subsection, we will show how the algorithm for the geometric distribution in Section~\ref{MCsub} can be used to draw samples from the conditional distribution $\mathbf{Y}\; | \; T(\mathbf{Y})=t$.

Note first that an argument like the one leading to (\ref{condprob}) gives the following expression:
\[
  \mathbf{P}(Y_1=y_1,\ldots,Y_n=y_n|\sum_{i=1}^n Y_i = t) = \frac{\prod_{i=1}^n \binom{y_i-r_i-1}{y_i}}{\binom{t+R-1}{t}}
\]
where $R=\sum_{i=1}^n r_i$. 

The following shows that we can sample from this conditional distribution by using the algorithm for the geometric distriburion. Note first that we can write for $i=1,\ldots,n$, 
\[
Y_i = \sum_{j=1}^{r_i} X_{ij}
\]
where $X_{ij}, \; i=1,\ldots,n; j=1,\ldots,r_i$ are i.i.d. from
 $\text{Geom}(p)$. 
Then 
	\begin{align}
	\mathbf{P}(\mathbf{Y}=\mathbf{y}\; |\; T(\mathbf{Y})=t)&= \mathbf{P}(Y_1=y_1, Y_2=y_2, \ldots ,Y_n=y_n\; | \; T(Y_1,Y_2,\ldots,Y_n)=t)\nonumber\\
	&= \mathbf{P}\left(\sum_{j=1}^{r_1}X_{1j}=y_1, \ldots ,\sum_{j=1}^{r_n}X_{nj}=y_n\; \Big| \; \sum_{i=1}^n\sum_{j=1}^{r_i}X_{ij}=t\right)\nonumber \\
	&= \displaystyle \mathop{\mathop{\sum_{(x_{ij}):\sum_j x_{1j}=y_1,}}_{,\ldots,}}_{,\sum_j x_{nj}=y_n}  \mathbf{P}\left( X_{ij}=x_{ij} \colon i=1,\ldots ,n, \ j=1,\ldots ,r_i \; \Big| \; \sum_{i,j}X_{ij}=t\right).\nonumber
	\end{align}
It follows from this that we can use the method for drawing samples from the conditional geometric distribution to draw condtional samples in the negative binomial case. More precisely, we can first draw a sample $x_1,x_2,\ldots x_R$, where $R=r_1+r_2+\ldots+r_n$, from the conditional distribution of $X_1,X_2,\ldots,X_{R}\; | \; T(\mathbf{X})=t,$
where $X_i\sim \text{Geom}(p)$ for $i=1,2,\ldots ,R$ and let 
\vekk{ 
$$y_1=\sum_{i=1}^{r_1}x_i, \ y_2=\sum_{i=1}^{r_2}x_{r_1+i}, \ldots ,y_n=\sum_{i=1}^{r_n}x_{r_1+\ldots+r_{n-1}+i}.$$ 
} 
$$y_1=\sum_{i=1}^{r_1}x_i, \ y_2=\sum_{i=r_1+1}^{r_1+r_2}x_{i}, \ldots ,y_n=\sum_{i=r_1+\ldots+r_{n-1}+1}^{R}x_{i}.$$ 
 We end up with a sample $\mathbf{y}$ from the desired conditional distribution $\mathbf{Y}\; | \; T(\mathbf{Y})=t$.
 
For simulation in practice, notice that Algorithm~\ref{norec} with input $t$ and $R$ gives numbers $k_1,k_2,\ldots,k_{R-1}$.  Using the transformation $\phi$ in Lemma~\ref{biject}, it is then seen that we have
\begin{eqnarray*}
y_1 &=& k_{r_1}-r_1 \\
y_2 &=& k_{r_1+r_2} - k_{r_1} - r_2 \\
&\vdots& \\
y_i &=& k_{r_1+\ldots+ r_{i}} - k_{r_1+\ldots+ r_{i-1}} - r_{i} \\
&\vdots& \\ 
y_{n-1} &=& k_{r_1+\ldots+ r_{n-1}} - k_{r_1+\ldots+ r_{n-2}} - r_{n-1}  \\ 
y_n &=& t - k_{r_1+\ldots+ r_{n-1}} + R - r_n .
\end{eqnarray*}

\subsection*{Some discrete distributions}
\begin{table}[H]
\begin{center}
	\begin{tabular}{| c | c |}
		\hline
		$X\sim \text{NB}(r,p)$ & $\mathbf{P}(X=x)=\binom{x+r-1}{x}(1-p)^xp^r, \ x=0,1,2,\ldots$\\
		\hline
		$X\sim \text{Pois}(\lambda)$ & $\mathbf{P}(X=x)=\frac{\lambda^xe^{-\lambda}}{x!}, \ x=0,1,2,\ldots$\\
		\hline
		$X\sim\text{Geom}(p)$ & $\mathbf{P}(X=x)=p(1-p)^x, \ x=0,1,2,\ldots$\\
		\hline
		$X\sim\text{Bin}(n,p)$ & $\mathbf{P}(X=x)=\binom{n}{x}p^x(1-p)^{n-x}, \ x=0,1,2,\ldots,n$\\
		\hline
		$X\sim\text{Ber}(p)$ & $\mathbf{P}(X=x)=p^x(1-p)^{1-x}, \ x=0,1$\\
		\hline
		$X\sim\mathcal{W}(q, \beta)$ & $\mathbf{P}(X=x)=q^{x^\beta}-q^{(x+1)^\beta}, \ x=0,1,2,\ldots$\\
		\hline
		$\mathbf{X}\sim\text{Mult}(t, \pi_1,\pi_2,\ldots ,\pi_n)$ & $\mathbf{P}(\mathbf{X}=\mathbf{x})=\frac{n!\pi_1^{x_1}\pi_2^{x_2}\cdots \pi_n^{x_n}}{x_1!x_2!\cdots x_n!}, \ \sum_{i=1}^nx_i=t$\\
		\hline
		$X\sim \mathrm{BG}(\alpha,\beta)$ & $\mathbf{P}(X=x)= \frac{B(\alpha+1, x+\beta)}{B(\alpha, \beta)}, \ x=0,1,2,\ldots$\\
		\hline
	\end{tabular}
\end{center}
\caption{Distributions}
\end{table}

\bibliography{gofbib}

\begin{thebibliography}{}

\bibitem[\protect\citeauthoryear{Barlow, Marshall, Proschan, et~al.}{Barlow
  et~al.}{1963}]{barlow63}
Barlow, R.~E., A.~W. Marshall, F.~Proschan, et~al. (1963).
\newblock Properties of probability distributions with monotone hazard rate.
\newblock {\em The Annals of Mathematical Statistics\/}~{\em 34\/}(2),
  375--389.

\bibitem[\protect\citeauthoryear{Beltr{\'a}n-Beltr{\'a}n and
  O'Reilly}{Beltr{\'a}n-Beltr{\'a}n and O'Reilly}{2019}]{reillypoisson}
Beltr{\'a}n-Beltr{\'a}n, J.~I. and F.~J. O'Reilly (2019).
\newblock On goodness of fit tests for the {P}oisson, negative binomial and
  binomial distributions.
\newblock {\em Statistical Papers\/}~{\em 60\/}(1), 1--18.

\bibitem[\protect\citeauthoryear{Bracquemond, Cr\'etois, and
  Gaudoin}{Bracquemond et~al.}{2002}]{BCG}
Bracquemond, C., E.~Cr\'etois, and O.~Gaudoin (2002).
\newblock A comparative study of goodness-of-fit tests for the geometric
  distribution and application to discrete time reliability.
\newblock {\em Laboratoire Jean Kuntzmann, Applied Mathematics and Computer
  Science, Technical Report\/}.

\bibitem[\protect\citeauthoryear{Bracquemond and Gaudoin}{Bracquemond and
  Gaudoin}{2003}]{BG}
Bracquemond, C. and O.~Gaudoin (2003).
\newblock A survey on discrete lifetime distributions.
\newblock {\em International Journal of Reliability, Quality and Safety
  Engineering\/}~{\em 10\/}(01), 69--98.

\bibitem[\protect\citeauthoryear{D'Agostino and Stephens}{D'Agostino and
  Stephens}{1986}]{agostino}
D'Agostino, R.~B. and M.~A. Stephens (Eds.) (1986).
\newblock {\em Goodness-of-fit Techniques}.
\newblock New York, NY, USA: Marcel Dekker, Inc.

\bibitem[\protect\citeauthoryear{Feller}{Feller}{1968}]{feller1}
Feller, W. (1968).
\newblock {\em An Introduction to Probability Theory and Its Applications},
  Volume~1.
\newblock New York: Wiley.

\bibitem[\protect\citeauthoryear{Fisher}{Fisher}{1950}]{fisher1950}
Fisher, R.~A. (1950).
\newblock The significance of deviations from expectation in a {P}oisson
  series.
\newblock {\em Biometrics\/}~{\em 6\/}(1), 17--24.

\bibitem[\protect\citeauthoryear{Gonz{\'a}lez-Barrios, O'Reilly, and
  Rueda}{Gonz{\'a}lez-Barrios et~al.}{2006}]{gonzalez}
Gonz{\'a}lez-Barrios, J.~M., F.~O'Reilly, and R.~Rueda (2006).
\newblock Goodness of fit for discrete random variables using the conditional
  density.
\newblock {\em Metrika\/}~{\em 64\/}(1), 77--94.

\bibitem[\protect\citeauthoryear{Heller}{Heller}{1986}]{heller1986}
Heller, B. (1986).
\newblock A goodness-of-fit test for the negative binomial distribution
  applicable to large sets of small samples.
\newblock In {\em Developments in Water Science}, Volume~27, pp.\  215--220.
  Elsevier.

\bibitem[\protect\citeauthoryear{Kyriakoussis, Li, and
  Papadopoulos}{Kyriakoussis et~al.}{1998}]{kyria}
Kyriakoussis, A., G.~Li, and A.~Papadopoulos (1998).
\newblock On characterization and goodness-of-fit test of some discrete
  distribution families.
\newblock {\em Journal of Statistical Planning and Inference\/}~{\em 74\/}(2),
  215--228.

\bibitem[\protect\citeauthoryear{Lindqvist, Erlemann, and Taraldsen}{Lindqvist
  et~al.}{2020}]{CMCr}
Lindqvist, B.~H., R.~Erlemann, and G.~Taraldsen (2020).
\newblock Conditional {M}onte {C}arlo revisited.
\newblock {\em Submitted and under revision\/}.

\bibitem[\protect\citeauthoryear{Lindqvist and Taraldsen}{Lindqvist and
  Taraldsen}{2005}]{LT05}
Lindqvist, B.~H. and G.~Taraldsen (2005).
\newblock Monte {C}arlo conditioning on a sufficient statistic.
\newblock {\em Biometrika\/}~{\em 92\/}(2), 451--464.

\bibitem[\protect\citeauthoryear{Lockhart, O'Reilly, and Stephens}{Lockhart
  et~al.}{2009}]{lockhartreilly}
Lockhart, R.~A., F.~O'Reilly, and M.~Stephens (2009).
\newblock Exact conditional tests and approximate bootstrap tests for the von
  mises distribution.
\newblock {\em Journal of Statistical Theory and Practice\/}~{\em 3\/}(3),
  543--554.

\bibitem[\protect\citeauthoryear{Lockhart, O'Reilly, and Stephens}{Lockhart
  et~al.}{2007}]{lockhartgibbs}
Lockhart, R.~A., F.~J. O'Reilly, and M.~A. Stephens (2007).
\newblock Use of the gibbs sampler to obtain conditional tests, with
  applications.
\newblock {\em Biometrika\/}~{\em 94\/}(4), 992--998.

\bibitem[\protect\citeauthoryear{Nakagawa and Osaki}{Nakagawa and
  Osaki}{1975}]{naka}
Nakagawa, T. and S.~Osaki (1975).
\newblock The discrete {W}eibull distribution.
\newblock {\em IEEE Transactions on Reliability\/}~{\em 24\/}(5), 300--301.

\bibitem[\protect\citeauthoryear{Ozonur, G\"okpinar, G\"okpinar, and
  Bayrak}{Ozonur et~al.}{2013}]{ODG}
Ozonur, D., E.~G\"okpinar, F.~G\"okpinar, and H.~Bayrak (2013).
\newblock Comparisons of the goodness of fit tests for the geometric
  distribution.
\newblock {\em Gazi University Journal of Science\/}~{\em 26\/}(3), 369--375.

\bibitem[\protect\citeauthoryear{Paul}{Paul}{2005}]{paul}
Paul, S.~R. (2005).
\newblock Testing goodness of fit of the geometric distribution: an application
  to human fecundability data.
\newblock {\em Journal of Modern Applied Statistical Methods\/}~{\em 4\/}(2),
  8.

\bibitem[\protect\citeauthoryear{Puig and Wei{\ss}}{Puig and
  Wei{\ss}}{2020}]{puig}
Puig, P. and C.~H. Wei{\ss} (2020).
\newblock Some goodness-of-fit tests for the {P}oisson distribution with
  applications in biodosimetry.
\newblock {\em Computational Statistics \& Data Analysis\/}~{\em 144}, 106878.

\bibitem[\protect\citeauthoryear{Rao and Chakravarti}{Rao and
  Chakravarti}{1956}]{rao1956}
Rao, C.~R. and I.~Chakravarti (1956).
\newblock Some small sample tests of significance for a {P}oisson distribution.
\newblock {\em Biometrics\/}~{\em 12\/}(3), 264--282.

\bibitem[\protect\citeauthoryear{Rueda, Reilly, and Perez-Abreu}{Rueda
  et~al.}{1991}]{RPoisson}
Rueda, R., F.~O. Reilly, and V.~Perez-Abreu (1991).
\newblock Goodness of fit for the {P}oisson distribution based on the
  probability generating function.
\newblock {\em Communications in Statistics - Theory and Methods\/}~{\em
  20\/}(10), 3093--3110.

\bibitem[\protect\citeauthoryear{Singh, Pudir, and Maheshwari}{Singh
  et~al.}{2014}]{singh}
Singh, B., P.~Pudir, and S.~Maheshwari (2014).
\newblock Parameter estimation of beta-geometric model with application to
  human fecundability data.
\newblock {\em arXiv preprint arXiv:1405.6392\/}.

\bibitem[\protect\citeauthoryear{Spinelli and Stephens}{Spinelli and
  Stephens}{1997}]{LockhartPoisson}
Spinelli, J.~J. and M.~A. Stephens (1997).
\newblock Cram\'er-von {M}ises tests of fit for the {P}oisson distribution.
\newblock {\em The Canadian Journal of Statistics / La Revue Canadienne de
  Statistique\/}~{\em 25\/}(2), 257--268.

\bibitem[\protect\citeauthoryear{Vila, Nakano, and Saulo}{Vila
  et~al.}{2019}]{vila2019}
Vila, R., E.~Y. Nakano, and H.~Saulo (2019).
\newblock Theoretical results on the discrete {W}eibull distribution of
  {N}akagawa and {O}saki.
\newblock {\em Statistics\/}~{\em 53\/}(2), 339--363.

\bibitem[\protect\citeauthoryear{Weinberg and Gladen}{Weinberg and
  Gladen}{1986}]{weinberg}
Weinberg, C.~R. and B.~C. Gladen (1986).
\newblock The beta-geometric distribution applied to comparative fecundability
  studies.
\newblock {\em Biometrics\/}, 547--560.

\bibitem[\protect\citeauthoryear{Wilf}{Wilf}{1999}]{Herb}
Wilf, H.~S. (1999).
\newblock East side, west side . . . - an introduction to combinatorial
  families-with maple programming.
\newblock "\url{https://www.math.upenn.edu/~wilf/lecnotes.html}".

\end{thebibliography}

\bibliographystyle{chicago}

 \end{document}